\documentclass[a4, 11pt] {amsart} 
\usepackage{mathrsfs}
\usepackage{amsmath}
\usepackage{amssymb}
\usepackage[all]{xy}
\usepackage[dvips]{graphicx, color}
\usepackage{amsmath}
\usepackage{amssymb}
\usepackage[all]{xy}

\usepackage{graphicx}

\begin{document}

\numberwithin{equation}{section}
\newtheorem{thm}[equation]{Theorem}
\newtheorem{pro}[equation]{Proposition}
\newtheorem{prob}[equation]{Problem}
\newtheorem{qu}[equation]{Question}
\newtheorem{cor}[equation]{Corollary}
\newtheorem{con}[equation]{Conjecture}
\newtheorem{lem}[equation]{Lemma}
\theoremstyle{definition}
\newtheorem{ex}[equation]{Example}
\newtheorem{defn}[equation]{Definition}
\newtheorem{ob}[equation]{Observation}
\newtheorem{rem}[equation]{Remark}

\hyphenation{homeo-morphism} 

\newcommand{\calA}{\mathcal{A}} 
\newcommand{\calD}{\mathcal{D}} 
\newcommand{\calE}{\mathcal{E}}
\newcommand{\calC}{\mathcal{C}} 
\newcommand{\Set}{\mathcal{S}et\,} 
\newcommand{\Top}{\mathcal{T}\!op \,}
\newcommand{\Topst}{\mathcal{T}\!op\, ^*}
\newcommand{\calK}{\mathcal{K}} 
\newcommand{\calO}{\mathcal{O}} 
\newcommand{\calS}{\mathcal{S}} 
\newcommand{\calT}{\mathcal{T}} 
\newcommand{\Z}{{\mathbb Z}}
\newcommand{\C}{{\mathbb C}}
\newcommand{\Q}{{\mathbb Q}}
\newcommand{\R}{{\mathbb R}}
\newcommand{\N}{{\mathbb N}}
\newcommand{\F}{{\mathcal F}}

\hfill

\title{The Hilali conjecture on product of spaces}  

\author{Shoji Yokura}

\date{}

\address{Department of Mathematics  and Computer Science, Graduate School of Science and Engineering, Kagoshima University, 1-21-35 Korimoto, Kagoshima, 890-0065, Japan
}

\email{yokura@sci.kagoshima-u.ac.jp}

\maketitle 

\begin{abstract}  The Hilali conjecture claims that a simply connected rationally elliptic space $X$ satisfies the inequality $\dim (\pi_*(X)\otimes \Q ) \leqq \dim H_*(X;\Q )$. In this paper we show that for any such space $X$ there exists a positive integer $n_0$ such that for any $n \geqq n_0$ the \emph{strict inequality $\dim (\pi_*(X^n)\otimes \Q ) < \dim H_*(X^n;\Q )$} holds, where $X^{n}$ is the product of $n$ copies of $X$.

\end{abstract}

\section{Introduction}

The most important and fundamental topological invariant in geometry and topology is the Euler--Poincar\'e characteristic $\chi(X)$,
which is defined to be the alternating sum of the Betti numbers $\beta_i(X):=\dim H_i(X;\Q)$:
$$\chi(X):= \sum_{i \geqq 0} (-1)^i\beta_i(X),$$
provided that each $\beta_i(X)$ and $\chi(X)$ are both finite.
Similarly, for a topological space whose fundamental group is an Abelian group one can define the \emph{homotopical Betti number} $\beta^{\pi}_i(X):= \dim (\pi_i(X)\otimes \Q)$ where $i\geqq 1$ and the \emph{homotopical Euler--Poincar\'e characteristic}:
$$\chi^{\pi}(X):= \sum_{i \geqq 1}  (-1)^i\beta^{\pi}_i(X),$$
provided that each $\beta^{\pi}_i(X)$ and $\chi^{\pi}(X)$ are both finite.
The Euler--Poincar\'e characteristic is the special value of the Poincar\'e polynomial $P_X(t)$ at $t=-1$ and the homotopical Euler--Poincar\'e characteristic is the special value of the homotopical Poincar\'e polynomial $ P^{\pi}_X(t)$ at $t=-1$:
$$P_X(t):= \sum_{i \geqq 0}  t^i \beta_i(X), \quad \chi(X) = P_X(-1),$$
$$ P^{\pi}_X(t):= \sum_{i \geqq 1} t^i \beta^{\pi}_i(X), \quad \chi^{\pi}(X) = P^{\pi}_X(-1).$$

From now on any topological space is assumed to be path-connected, unless otherwise stated. The well-known Hilali conjecture \cite{Hil} claims that for a simply connected rationally elliptic space $X$ (i.e.., both $\dim (\pi_*(X)\otimes \Q ) < \infty$ and $\dim H_*(X;\Q ) < \infty$), then 
$$\dim (\pi_*(X)\otimes \Q ) \leqq \dim H_*(X;\Q ), \quad \text{namely}, \quad P^{\pi}_X(1) \leqq P_X(1).$$
Here 
$\pi_*(X)\otimes \Q := \sum_{i \geqq 1} \pi_i(X) \otimes \Q$  and  $H_*(X;\Q ):= \sum_{i \geqq 0} H_i(X;\Q)$. 
No counterexample to the Hilali conjecture has been so far found yet.

In this paper we show that for a simply connected rationally elliptic space $X$
there exists a positive integer $n_{0}$ such that for $\forall \, \, n \geqq n_{0}$
$$\dim (\pi_*(X^n)\otimes \Q ) < \dim H_*(X^{n};\Q ).$$
Here $X^n$ is the product $X^n =\underbrace{X \times \cdots \times X}_{n}$.

\begin{rem} In the Hilali conjecture a topological space $X$ is required to be simply connected, i.e., the fundamental group is trivial, $\pi_1(X)=0$. One of the reason of this requirement is that in general the fundamental group is not an Abelian group, thus one cannot define tensoring  $\pi_1(X) \otimes \mathbb Q$. In fact, you do need the condition of simply connectedness. The well-known counterexample is $S^1 \vee S^2$. Since $H_0(S^{1} \vee S^{2}) =H_1(S^{1} \vee S^{2})=H_1(S^{1} \vee S^{2}) =\mathbb Z$, hence $\dim H_*(S^{1} \vee S^{2}; \Q)=3$. However $\pi_{1}(S^{1}\vee S^{2})=\mathbb Z$ (thus $S^{1} \vee S^{2}$ is not simply connected) and $\pi_{2}(S^1 \vee S^{2})=\mathbb Z^{\infty}$. So $\dim \pi_{*}(X) \otimes \Q= \infty$. Hence $S^{1} \vee S^{2}$ is not rationally elliptic. So, a naive question is if the Hilali conjecture still holds for a rationally elliptic space whose fundamental group is an Abelian group.

\end{rem}
\section{Poincar\'e polynomial and homotopical Poincar\'e polynomial}

The Poincar\'e polynomial $P_X(t)$ is \emph{multiplicative} in the following sense:
$$P_{X \times Y}(t) = P_X(t) \times P_Y(t),$$
which follows from the K\"unneth Formula:
$$H_n(X \times Y;\Q)= \sum_{i+j=n} H_i(X; \Q) \otimes H_j(Y;\Q).$$
The homotopical Poincar\'e polynomial $P^{\pi}_X(t)$ is \emph{additive} in the following sense:
$$P^{\pi}_{X \times Y}(t) = P^{\pi}_X(t) + P^{\pi}_Y(t),$$
which follows from 
$$\pi_i(X \times Y) = \pi_i (X) \times \pi_i(Y) =\pi_i(X) \oplus \pi_i(X) \quad
\text{and} \quad  (A \oplus B)\otimes \Q = (A \otimes \Q ) \oplus (B \otimes \Q).$$

Here, for a later use, we compute the Poincar\'e polynomial and homotopical Poincar\'e polynomial of spheres.
The following are well-known results (due to Serre Finiteness Theorem \cite{Se0, Se}):
$$\pi_i(S^{2k}) \otimes \Q 
=\begin{cases} 
\Q & \, i=2k\\
 \Q & \, i=4k-1\\
\, 0  & \,  i\not =2k, 4k-1
\end{cases}  \qquad 
\pi_i(S^{2k+1}) \otimes \Q 
=\begin{cases} 
 \Q & \,  i=2k+1\\
\, 0 & \, i \not =2k+1
\end{cases} 
$$
Hence we have that 
\begin{equation}\label{equ1}
P^{\pi}_{S^{2k+1}}(t) = t^{2k+1} \, \text{ and} \, P_{S^{2k+1}}(t) = t^{2k+1} +1.
\end{equation}
Thus we have
$$P^{\pi}_{S^{2k+1}}(t) < P_{S^{2k+1}}(t) \, \ \text{for} \, \,   \forall t.$$
However we have 
\begin{equation}\label{equ2}
P^{\pi}_{S^{2k}}(t) = t^{4k-1} + t^{2k} \, \text{and} \, P_{S^{2k}}(t) = t^{2k} +1.
\end{equation}
 Hence we have
$$\begin{cases}
P^{\pi}_{S^{2k}}(t) < P_{S^{2k}}(t), & \, t <1 \\
P^{\pi}_{S^{2k}}(t) = P_{S^{2k}}(t), & \,t=1 \\
P^{\pi}_{S^{2k}}(t) > P_{S^{2k}}(t), & \, t >1 \\
\end{cases}
$$

\section{The Hilali conjecture on products of spheres}

It is known that the Hilali conjecture holds for products of spheres. In this section, first we show the following more general statement,
using the multiplicativity of the Poincar\'e polynomial and the additivity of the homotopical Poincar\'e polynomial observed above.

First we observe that for a pathconnected space $X$ $\dim H_{*}(X; \Q) =1$ if and only if $X$ is rationally homotopy equivalent to a point, which is due to the fundamental fact that $X \sim_{\Q} Y$ if and only if the Sullivan's minimal models $M_{X}$ and $M_{Y}$ are isomorphic. Another simpler argument is using the well-known Whitehead--Serre Theorem \cite[Theorem 8.6]{FHT}. Indeed, $\dim H_{*}(X; \Q) =1$ for a pathconnected space X is equivalent to $(a_X)_*:H_*(X;\Q) \to 
H_*(pt)=\Q$ being an isomorphism, where $a_X:X \to pt$ is the map to a point. Thus it follows from the Whitehead--Serre Theorem that $(a_X)_*:\pi_*(X)\otimes \Q \to \pi_*(pt)\otimes \Q =0$ is an isomorphism. Therefore we get the following strict inequality
\begin{equation}\label{equ0}
0 = \dim (\pi_{*}(X) \otimes \Q) < \dim H_{*}(X; \Q)=1, \quad \text{namely} \quad 0=P^{\pi}_{X}(1) < P_{X}(1) =1
\end{equation}

\begin{pro} \label{pro1} Let $X_i \, (1 \leqq i \leqq n)$ be a rationally elliptic space such that the fundamental group is an Abelian group, then $X_1 \times \cdots \times X_n$ is also rationally elliptic, and if $P^{\pi}_{X_i}(1) \leqq P_{X_i}(1)$, then $P^{\pi}_{X_1 \times \cdots \times X_n}(1) \leqq P_{X_1 \times \cdots \times X_n}(1)$.
\end{pro}

The above observation (\ref{equ0}) implies that if $X_i$ ($1 \leqq i \leqq k$) satisfies $P_{X_{i}}(1)=\dim H_{*}(X_{i};\Q) \geqq 2$ and $X_j$ ($k+1 \leqq j \leqq n$) satisfies $P_{X_{j}}(1)=\dim H_{*}(X_{j};\Q) =1$, thus $P^{\pi}_{X_{j}}(1)=0$, then we have
$$P^{\pi}_{X_1 \times \cdots \times X_n}(1) = P^{\pi}_{X_1 \times \cdots \times X_k}(1)+ P^{\pi}_{X_{k+1}}(1) + \cdots +P^{\pi}_{X_n}(1)= P^{\pi}_{X_1 \times \cdots \times X_k}(1).$$
$$P_{X_1 \times \cdots \times X_n}(1) = P_{X_1 \times \cdots \times X_k}(1) \times P_{X_{k+1}}(1) \times \cdots \times P_{X_n}(1)= P_{X_1 \times \cdots \times X_k}(1).$$
(Namely, taking the product by a space $Z$ satisfying $\dim H_{*}(Z; \Q)=1$ does not change the value of $P^{\pi}_{\bullet}(1)$ and $P_{\bullet}(1)$.)
Therefore, to prove the proposition, we can assume that for each $X_i$ we have $P_{X_{i}}(1) \geqq 2$.
Since $P^{\pi}_{X_1 \times \cdots \times X_n}(1)=P^{\pi}_{X_1}(1) + \cdots +P^{\pi}_{X_n}(1)$ and  $P_{X_1 \times \cdots \times X_n}(1)= P_{X_1}(1)\cdots P_{X_n}(1)$, this proposition follows from the following elementary lemma.
\begin{lem} Let $a_i, b_i \, (1 \leqq i \leqq n)$ be real numbers such that $a_i \leqq b_i$ and $2 \leqq b_i$ for each $i$. Then we have
$$a_1 + a_2  \cdots + a_n \leqq b_1b_2 \cdots b_n.$$
In particular, we have 
$$b_1 + b_2 + \cdots +b_n \leqq b_1b_2 \cdots b_n.$$
\end{lem}
\begin{proof} A proof is easy, but for the sake of completeness we write a proof.
Let $n=2$. Since $a_1 \leqq b_1$ and $a_2 \leqq b_2$, we have $a_1 + a_2  \leqq b_1+ b_2$.
\begin{align*}
b_1b_2 -(b_1+b_2) &= (b_1b_2 -b_1-b_2+1) -1 \\
& =(b_1 -1)(b_2-1) -1 \\
& \geqq 1\cdot 1-1 =0 \quad (\text{because $b_i \geqq 2$, thus $b_i -1 \geqq 1$})
\end{align*}
Hence $b_1 +b_2 \leqq b_1 b_2.$ Therefore we have $a_1 +a_2 \leqq b_1b_2.$ Now, suppose that we have
$a_1+a_2+\cdots +a_{n-1} \leqq b_1b_2 \cdots b_{n-1}.$ Since $2 \leqq b_1b_2 \cdots b_{n-1}$, by applying $a_1+a_2 \leqq b_1b_2$ to 
the inequalities $a_1+a_2+\cdots +a_{n-1} \leqq b_1b_2 \cdots b_{n-1}$ and $a_n \leqq b_n$ we get
$$(a_1+a_2+ \cdots +a_{n-1})+a_n \leqq (b_1b_2 \cdots b_{n-1})b_n,$$
namely we get $a_1 + a_2  \cdots + a_n \leqq b_1b_2 \cdots b_n$.
\end{proof}

As an application of the above proposition we can see that the Hilali conjecture holds for the product of a finite family of spheres of dimension $\geqq 2$. (Note that we need that the dimension of each sphere is $\geqq 2$ because in the Hilali conjecture a space has to be simply connected.) It follows from (\ref{equ1}) and (\ref{equ2}) above that for any sphere $S^n$ of any dimension $n$ we have
$$P^{\pi}_{S^n}(1) \leqq P_{S^n}(1) \quad  \text{and} \quad  P_{S^n}(1)=2.$$
Therefore the following corollary follows from Proposition \ref{pro1}:
\begin{cor} The Hilali conjecture holds for the product $S^{n_1} \times S^{n_2} \times \cdots \times S^{n_r}$ of any finite family of spheres $S^{n_i}$ of any dimension $\geqq 2$:
$$P^{\pi}_{S^{n_1} \times S^{n_2} \times \cdots \times S^{n_r}}(1) \leqq  P_{S^{n_1} \times S^{n_2} \times \cdots \times S^{n_r}}(1).$$
Hence the Hilali conjecture also holds for any topological space $S$ which is homotopy equivalent to such a product of spheres:
$$P^{\pi}_{S}(1) \leqq P_{S}(1).$$
\end{cor}
\begin{rem} It is clear that the above inequalities hold even if we do not require the dimension of each sphere to be $\geqq 2$.
\end{rem}

We can get the following corollary from the above Proposition \ref{pro1} and the above remark:
\begin{cor}\label{cor1} Let $X$ be a rationally elliptic space such that its fundamental group is an Abelian group. Then $X^n:=X \times X \times \cdots \times X$ is also rationally elliptic, and if $P^{\pi}_X(1) \leqq P_X(1)$, then $P^{\pi}_{X^n}(1) \leqq P_{X^n}(1)$.
\end{cor}
\section{Hilali conjecture ``modulo product"}

Motivated by the above Corollary \ref{cor1}, from the multiplicativity of the Poincar\'e polynomial $P_X(1)$ and the additivity of homotopical Poincar\'e polynomial $P^{\pi}_X(1)$ we can get the following theorem from an elementary calculus fact:
\begin{thm}[\emph{Hilali conjecture ``modulo product"}] Let $X$ be a  
rationally elliptic space such that its fundamental group is an Abelian group. Then there exists some integer $n_0$ such that for $\forall \, \, n\geqq n_0$ the following strict inequality holds:
$$\dim (\pi_*(X^n)\otimes \Q ) < \dim H_*(X^{n};\Q ), i.e., P^{\pi}_{X^n}(1) < P_{X^n}(1).$$
\end{thm}
\begin{proof} 
Since $X$ is 
rationally elliptic, for any integer $X^n$ is also 
rationally elliptic.

If $P_{X}(1)=1$, then it follows from (\ref{equ0}) that $P^{\pi}_{X}(1)=0$, hence for any integer $n \geqq 1$ we have 
$$0= P^{\pi}_{X^n}(1) < P_{X^n}(1) = 1.$$
So, suppose that $P_{X}(1) \geqq 2$.
The multiplicativity of the Poincar\'e polynomial $P_X(1)$ and the additivity of homotopical Poincar\'e polynomial $P^{\pi}_X(1)$ 
imply the following
$$\frac{P^{\pi}_{X^n}(1)}{P_{X^n}(1)} = \frac{nP^{\pi}_X(1)}{(P_X(1))^n}.$$
Since $P_X(1) \geqq 2$, $\frac{1}{P_X(1)} <1$. Thus it follows from the elementary calculus\footnote{If $|r| <1$, we have $\displaystyle \lim_{n \to \infty} nr^n=0$.} that
$$\lim_{n \to \infty} n \Bigl (\frac{1}{P_X(1)} \Bigr)^n = 0.$$
Therefore, whatever the value $P^{\pi}_X(1)$ is, we obtain
$$\lim_{n \to \infty} n P^{\pi}_X(1)\Bigl (\frac{1}{P_X(1)} \Bigr)^n = \lim_{n \to \infty} \frac{nP^{\pi}_X(1)}{(P_X(1))^n} = 0.$$
Hence there exists some integer $n_0$ such that for all $n \geqq n_0$
$$\frac{P^{\pi}_{X^n}(1)}{P_{X^n}(1)}  = \frac{nP^{\pi}_X(1)}{(P_X(1))^n} < 1.$$
Therefore there exists some integer $n_0$ such that for all $n \geqq n_0$
$$P^{\pi}_{X^n}(1) < P_{X^n}(1).$$
\end{proof}
\begin{defn} Let $X$ be a rationally elliptic space such that its fundamental group is an Abelian group. 
\begin{enumerate}
\item The minimum integer $n_0$ such that
$$P^{\pi}_{X^{n_0}}(1) < P_{X^{n_0}}(1)$$
shall be called the \emph{``homology-rank-dominating" power} of $X$ and denoted by $\frak p_0(X)$.
\item The minimum integer $n_0$ such that
$$P^{\pi}_{X^{n_0}}(1) \leqq P_{X^{n_0}}(1)$$
shall be called the \emph{``homology-rank-almost-dominating" power} of $X$ and denoted by $\frak p(X)$.
\end{enumerate}
\end{defn}
\begin{rem} 
Since $P^{\pi}_{X^{n_0}}(1)$ and $P_{X^{n_0}}(1)$ are both homotopy invariant, the powers   $\frak p_0(X)$ and  $\frak p(X)$ are both homotopy invariants
\end{rem}
\begin{rem} 
The Hilali conjecture claims that the ``homology-rank-almost-dominating" power $\frak p(X)$ is always $1$.
\end{rem}
\begin{ex} Let $n$ be any positive integer. Since $P^{\pi}_{S^{2n}}(1) = 2$ and $P_{S^{2n}}(1) = 2$, we have
$$P^{\pi}_{S^{2n} \times S^{2n}}(1) =2+2=4, P_{S^{2n} \times S^{2n}}(1) = 2\times 2 =4,$$
$$P^{\pi}_{S^{2n} \times S^{2n} \times S^{2n}}(1) =2 +2 +2 =6, P_{S^{2n} \times S^{2n}\times S^{2n}}(1)=2 \times 2 \times 2 =8.$$
Hence we have 
$$\frak p_0(S^{2n})=3 \quad \text{and} \quad \frak p(S^{2n})=1.$$
\end{ex}
\begin{ex} Let $n, m$ be any positive integers. Since $P^{\pi}_{S^{2n} \times S^{2m}}(1) = 2+2=4$ and $P_{S^{2n} \times S^{2m}}(1) = 2+2=4$, we have
$$P^{\pi}_{(S^{2n} \times S^{2m}) \times (S^{2n} \times S^{2m}) }(1) = 4 +4 =8, P_{(S^{2n} \times S^{2m}) \times (S^{2n} \times S^{2m}) }(1) =4 \times 4 =16.$$
Hence we have 
$$\frak p_0(S^{2n} \times S^{2m})=2 \quad \text{and} \quad \frak p(S^{2n} \times S^{2m})=1.$$
\end{ex}
\begin{rem} If the Hilali conjecture is correct, then we have the following:
\begin{enumerate}
\item if $\dim (\pi_*(X) \otimes \Q) <\dim H_{*}(X;\Q)$, then $\frak p_0(X)=1.$
\item if $\dim (\pi_*(X) \otimes \Q) =\dim H_{*}(X;\Q)=2$, then $\frak p_0(X)=3.$
\item if $\dim H_{*}(X;\Q) \geqq 3$ and $\dim (\pi_*(X) \otimes \Q) =\dim H_{*}(X;\Q)$, then $\frak p_0(X)=2.$
\end{enumerate}
\end{rem}
\section{A final remark}
If we plug in $t=-1$ in the two equations (\ref{equ1}) and (\ref{equ2}) in \S2, we get the following
\begin{equation}\label{odd sphere}
P^{\pi}_{S^{2k+1}}(-1) = -1 <  P_{S^{2k+1}}(-1) = 0, 
\end{equation}
\begin{equation}\label{even sphere}
P^{\pi}_{S^{2k}}(-1) = 0 <  P_{S^{2k}}(-1) = 2.
\end{equation}
Therefore, for the product $S^{n_1} \times S^{n_2} \times \cdots \times S^{n_r}$ of any finite family of spheres $S^{n_i}$ we have 
\begin{equation}
P^{\pi}_{S^{n_1} \times S^{n_2} \times \cdots \times S^{n_r}}(-1)  \leqq 0 \quad \text{and} \quad  0\leqq P_{S^{n_1} \times S^{n_2} \times \cdots \times S^{n_r}}(-1).
\end{equation}
Thus we get the following :
\begin{cor}\label{-1} For the product $S^{n_1} \times S^{n_2} \times \cdots \times S^{n_r}$ of any finite  family of spheres $S^{n_i}$ we have
$$P^{\pi}_{S^{n_1} \times S^{n_2} \times \cdots \times S^{n_r}}(-1) \leqq  P_{S^{n_1} \times S^{n_2} \times \cdots \times S^{n_r}}(-1).$$
Hence for any topological space $S$ which is homotopy equivalent to such a product of spheres, we have
$$P^{\pi}_{S}(-1)\leqq P_{S}(-1).$$
\end{cor}

By the definition it is clear that $P^{\pi}_X(0)<P_X(0)$. Since $-1$ is symmetric to $1$ as to $0$, the above Corollary \ref{-1} is in fact a special case of the following proposition, which is a \emph{``mirror version"} of the Hilali conjecture:
\begin{pro}(\cite[Proposition 32.10]{FHT}) Let $X$ be a simply connected rationally elliptic space. Then 
$$\chi^{\pi}(X) \leqq 0 \quad \text{and} \quad 0\leqq \chi(X).$$
Namely we have $P^{\pi}_X(-1) \leqq P_X(-1)$ i.e., 
$$ \sum_{i\geqq 1} (-1)^i \dim \Bigl ( \pi_i(X)\otimes \Q \Bigr) \leqq \sum_{i\geqq 0} (-1)^i \dim H_i(X;\Q ).$$
\end{pro}

\begin{rem} It should be noted that in \cite[Proposition 32.10]{FHT} a bit stronger statement is also given; $\chi^{\pi}(X) =0 \Longleftrightarrow \chi(X)>0$, in other words $\chi^{\pi}(X) <0 \Longleftrightarrow  \chi(X)=0$. Indeed, for example, let us consider the above product of finitely many spheres $X:=S^{n_1} \times S^{n_2} \times \cdots \times S^{n_r}$, in which case $\chi^{\pi}(X) =0 \Longleftrightarrow \chi(X)>0$ follows from both (\ref{odd sphere}) and (\ref{even sphere}).
\end{rem}
\begin{rem} Finally we remark that in \cite{YY} we discuss the ratio $h(X):=\frac{P^{\pi}_X(1)}{P_X(1)}$ of a fibration $X$ of elliptic spaces. Note that if the Hilali conjecture holds, then $0 \leqq h(X) \leqq 1$.\\
\end{rem}
{\bf Acknowledgements:} The author would like to thank Toshihiro Yamaguchi for useful comments. This work is supported by JSPS KAKENHI Grant Numbers JP16H03936 and JP19K03468. \\

\end{document}